\documentclass[11pt,a4paper,reqno]{amsart}
\usepackage{datetime}
\usepackage{amsmath}
\usepackage{amssymb,latexsym}
\usepackage[dvips]{graphicx}
\usepackage{color}
\usepackage{cite}
\usepackage{amsthm}

\usepackage{etoolbox}
\makeatletter
\patchcmd{\@maketitle}
  {\ifx\@empty\@dedicatory}
  {\ifx\@empty\@date \else {\vskip3ex \centering\footnotesize\@date\par\vskip1ex}\fi
   \ifx\@empty\@dedicatory}
  {}{}
\patchcmd{\@adminfootnotes}
  {\ifx\@empty\@date\else \@footnotetext{\@setdate}\fi}
  {}{}{}
\makeatother

\newcommand{\R}{\mathbb R}

\newtheorem{theorem}{Theorem} [section]
\newtheorem{lemma}{Lemma} [section]
\newtheorem{proposition}{Proposition} [section]

\newtheorem{definition}{Definition} [section]

\newtheorem{remark}{Remark}[section]

\let\ssection=\section\renewcommand{\section}{\setcounter{equation}{0}\ssection}

\begin{document}
\address{Mohamad Darwich: Faculty of Sciences, Laboratory of Mathematics, Doctoral School of Sciences and Technology, Lebanese University Hadat, Lebanon.}
\title[]{Global existence for the nonlinear fractional Schr\"odinger equation
with fractional dissipation.}
\date{\today}
\author{Mohamad Darwich.}

\begin{abstract}
We consider the initial value problem for the fractional nonlinear
Schr\"odinger equation with a fractional dissipation. Global existence and scattering are proved depending
on the order of the fractional dissipation.
\end{abstract}

\maketitle

%%%%%%%%%%%%%%%%%%%%%%%%%%%%%%%%%%%%%%%%%%%%%%%%%%%%%%%%%%%%%%%%%%%%%%%%%%%%%%%%%%%%%%%%%%%%%%%%%%%%%%%%%%%%%%%%%%%%%%%%%%%%%%%%%%%
\section{Introduction}
Consider the Cauchy problem for the damped fractional nonlinear Schr\"odinger equation
\begin{equation}
\left\{
\begin{array}{l}
%$$ \left\lbrace \begin{array}{l}
iu_{t} -(-\Delta)^{\alpha}u +|u|^{p-1}u + ia(-\Delta)^s u =0,  (t,x) \in [0,\infty[\times \mathbb{R}^{d}. \\
u(0)= u_{0} 
%\end{array} $$
\end{array}
\right. \label{NLSasalpha}
\end{equation}
  \hspace{0.3cm}where $a > 0$ is the coefficient of friction, $d\geq 2$, $\alpha \in (\frac{d}{2d-1},1)$, $ s>0$ with $L^2$-critical nonlinearity i.e $p=1+\frac{4\alpha}{d}$.\\
  
  In the classical case ($\alpha =1$ and $a=0$) equation (\ref{NLSasalpha}) arises in various areas of nonlinear optics, plasma physics and fluid mechanics to describe propagation phenomena in dispersive media.\\
  \\
 When $a=0$ and $0< \alpha < 1$ equation \eqref{NLSasalpha} (called FNLS : Fractional NLS) can be seen as a canonical model for
a nonlocal dispersive PDE with focusing nonlinearity that can exhibit solitary waves,
turbulence phenomena  which has been studied by many authors \cite{Alexandru}, \cite{David}, \cite{Hichem}, \cite{Cho}, \cite{sparber1}, \cite{krieger} and \cite{Weinstein1} in %List of studies of fractional NLS are treated by \cite, in
mathematics, numerics, and physics.
The FNLS equation is a fundamental equation of fractional
quantum mechanics, which was derived by Laskin \cite{Laskin1}, \cite{Laskin2} as a result of extending
the Feynman path integral, from the Brownian-like to L´evy-like quantum mechanical
paths. The Cauchy problem for FNLS was studied in \cite{Guo1} and \cite{Guo2} and proved that it is well-posed and scatters in the radial energy space and in \cite{Guo3} the author proves that the equation is globaly well posed  for small data.\\

%To take into account weak dissipation effects, one usually add a linear damping term as in the  linear damped NLS equation (see for instance Fibich \cite{FibichG} ):
%$$i u_t +\Delta u +ia  u + |u|^p  u =0 , a>0.$$

%However, in a wide range of situations a frequency-dependent attenuation has been observed (cf \cite{Chen}).  This motivates  to rather complete the NLS equation with a laplacian term as in the following complex Ginzburg-Landau equation studied in Passota-Sulem-Sulem \cite{passot}:

%$$i u_t +\Delta u -ia \Delta  u + |u|^p  u =0   , a>0,$$

%Now, in many cases of practical importance the damping cannot be  described by a  local term even in the long-wavelength limit. In media with dispersion the weak dissipation is, in general, non local (see for instance Ott-Sudan \cite{Ott}). It is thus quite natural to complete the NLS equation by a non local dissipative term in order to take into account some dissipation phenomena. 

In this this paper we complete the $L^2$-critical FNLS equation with a  fractional laplacian of order $ 2s $, $ s>0 $. The fractional laplacian is commonly used to model  fractal (anomalous) diffusion related to the L\'evy flights (see e.g. Stroock \cite{Stroock},  Bardos and all \cite{Bardos}, Hanyga \cite{Hanyga}). It also appears in the physical literature to model attenuation phenomena of acoustic waves in irregular porous random media
 (cf. Blackstock \cite{Black}, Gaul \cite{Gaul}, Chen-Holm \cite{Chen}).
 
 Note that for $s=0$, the global existence for \eqref{NLSasalpha} was proved in \cite{Saanouni} for a large daming tem ( i.e for large $a > 0$), in this paper we will obtain the global existence result for any damping term $a>0.$

Finally, the case $\alpha =1$ and a nonlinear damping of the type  $ia|u|^pu$,  has been studied by Antonelli-Sparber and Antonelli-Carles-Sparber (cf. \cite{Sparber} and \cite{Sparber2}). In this case the  origin of the nonlinear damping term  is multiphoton absorption. \vspace*{4mm}

The purpose of this paper is to prove some global well-posedness
and scattering  results for \eqref{NLSasalpha} in the radial case and the rest of the paper is organized as follows. Section 2 is devoted to prove the local existence results.  In section 3 we will show the main results i.e the global well-posedness of  equation \eqref{NLSasalpha} and the scattering. \\

Now let us define the following quantities:\\
$L^2$-norm : $m(u)=\left\|u\right\|_{L^2} = \Bigl( \displaystyle{\int |u(x) |^2dx}\Bigr)^{1/2}$.\\
Energy : $E(u) = \displaystyle{\frac{1}{2}\|(-\Delta)^{\frac{\alpha}{2}}u\|_{L^2}^{2} - \frac{d}{4\alpha + 2d}\|u\|_{L^{\frac{4\alpha}{d}+2}}^{{\frac{4\alpha}{d}+2}}}. $\\
%Kinetic momentum : $P(u)=Im(\displaystyle{\int} \nabla u(x) \overline{u}(x))%\, dx .$\\
%Now, for  (\ref{NLSasalpha}) with $ a>0$, there does not exist conserved quantities anymore.
However, it is easy to prove that if $u$ is a  smooth solution of (\ref{NLSasalpha}) on $[0,T[ $, then for all $ t\in [0,T[ $ it holds
\begin{equation}\label{masse}
\frac{d}{dt}\big( m(u(t))\big)= - a\|(-\Delta)^{\frac{s}{2}}u\|^2_{L^2};
\end{equation}
\begin{equation}\label{energie}
\frac{d}{dt} \big(E(u(t))\big) =  \displaystyle{-a\int |(-\Delta)^{\frac{s+\alpha}{2}}u(t)|^2 +a \Im \int ((-\Delta)^s u(t)) |u(t)|^{p-1}\overline{u}(t)}.
\end{equation}
%\begin{equation}\label{momentum}
%\displaystyle{\frac{1}{2} \frac{d}{dt} P(u(t)) = a Im \int ((-\Delta)^{s}u(t))\overline{\nabla u}(t)}.
%\end{equation}
%In  \cite{Darwich}, the first author  studied the case $ s=0$. He  proved the global existence  in $H^1$ for
%$\|u_0\|_{L^2} \leq \|Q\|_{L^2}$, and  showed that the log-log regime is stable by  such perturbations (i.e. there exist  solutions 
 % that blowup in finite time with the log-log law).\\
 %In  \cite{passot}, Passot, Sulem and  Sulem  proved that 
%the solutions  are global in $H^1(\mathbb{R}^{2})$ for $s=1$. However, their method does not seem to apply for any other values of $d$.\\

%of finite time blowup dynamics with the log-log blow-up speed for $\|\nabla u \|_{L^2}$.
%\vspace{0,3cm}
 
Let us now state our results:

\begin{theorem}\label{theorem1}
Let $d\geq 2 $, $\alpha \in (\frac{d}{2d-1},1)$ and  $0 < s < \alpha$, such that $ s+\alpha \geq 1$ then there exists a real number $\beta > 0$  such that for any initial datum $u_0 \in H_{rd}^\alpha(\R^d) $ with $\|u_0\|_{L^2} <\beta $, the emanating solution u is global in $ H_{rd}^\alpha(\R^d)$.
\end{theorem}
\begin{theorem}\label{theorem2}
Let $d \geq 2$,  $\alpha \in (\frac{d}{2d-1},1)$ and $s \geq \alpha$. Then the Cauchy problem (\ref{NLSasalpha}) is globally well-posed in $ H_{rd}^\alpha(\R^d) $.
\end{theorem}
%\begin{theorem}\label{theorem4} Let $d\in \N^*$, $ 0<s<1 $ and $ a>0 $. 
%\begin{enumerate}
%\item There exists a real number $0< \gamma=\gamma(d)\le \|Q\|_{L^2} $ such that for any initial datum $u_0 \in H^1(\R^d) $ with $\|u_0\|_{L^2} <\gamma $, the emanating solution u is global in $ H^1$ with an energy that is non increasing.
%\item
% There does not exists any initial datum $u_0$, with $\|u_0\|_{L^2} \leq \|Q\|_{L^2}$, such that the solution $u$ of (\ref{NLSas})
 %blows up at finite time $ T^* $  and satisfies 
%$$\frac{1}{(T^*-t)^{\alpha }} \lesssim \|\nabla u(t)\|_{L^2(\mathbb{R}^{d})} \lesssim \frac{1}{(T^*-t)^{\beta }},\quad \forall \, 0<T-t\ll 1\, , 
%$$
%for some pair $(\alpha,\beta) $ satisfying  $0	<\beta <\frac{1}{2s}  $ and $  \beta(1+s) -1/2 <\alpha \le \beta $. 
%\end{enumerate}

%\end{theorem}
\begin{theorem}\label{third theorem}
Let $\alpha \in (\frac{d}{2d-1},1)$, $s = \alpha$,  $u_0 \in H_{rd}^\alpha(\mathbb{R}^{d})$ and $u \in C(\mathbb{R}^{+},H_{rd}^\alpha)$ be the
global solution to (\ref{NLSasalpha}). Then:
\begin{enumerate}
\item There exists $ u_{+} \in L^2$ such that $\|(u - S_{a,\alpha,s}(.)u_{+})(t)\|_{L^2} \longrightarrow 0$, as $t \longrightarrow +\infty$.
\item $\|u\|_{L^{\frac{4\alpha}{d}+2}(\R_+^*\times\R^d) } \longrightarrow 0, ~ \text{when}~ a \longrightarrow +\infty.$
\end{enumerate}
\end{theorem}

Acknowledgments : The author thanks Luc Molinet for his valuable remarks and comments in this paper.

%In this section, we prove Theorem \ref{theorem2} and  part (1) of Theorem \ref{theorem4}. Theorem \ref{theorem2} will follow from  an a priori estimate on the critical Strichartz norm whereas  part (1) of Theorem \ref{theorem4} follows from a monotonicity of the energy. 
%Finally to prove Theorem \ref{theorem2}, we etablish an a priori estimate on the critical Strichartz norm.
\section{Local existence result}
Recall that the main tools to prove the local existence results for  the FNLS equation are the Strichartz estimates for the associated linear propagator $ e^{i(-\Delta)^{\alpha} t}$. Let us mention that in the case $ a > 0$ the same results on the local Cauchy problem for (\ref{NLSasalpha}) can be established  in exactly the same way as in the case $a = 0$, since the same Strichartz estimates hold.
\subsection{Strichartz estimate}
\begin{definition}
 A pair $(q,r)$, $ q,r\geq 2$ is said to be 
admissible if:
$$\frac{4d+2}{2d-1}\leq q \leq \infty,~~ \frac{2}{q} + \frac{2d-1}{r} \leq d-\frac{1}{2},$$\\
or\\
$$2 \leq \frac{4d+2}{2d-1},~~ \frac{2}{q} + \frac{2d-1}{r} < d-\frac{1}{2}.$$

\end{definition}

These Strichartz estimates read in the following proposition see \cite{Guo4}:
\begin{proposition}\label{strichartzforNLS}
Suppose $d\geq 2$, $a=0$ and  $u$ be a radial solution of \ref{NLSasalpha},  then  for every admissible pair $(q,r)$ satisfie the following condition:
$$
\frac{2}{q} + \frac{d}{r}=\frac{d}{2} - \gamma,~~ \frac{2}{\tilde{q}} + \frac{d}{\tilde {r}}=\frac{d}{2} + \gamma
$$ 
it holds:\\
$$\|u\|_{(L^q_t L^r_x \cap L^\infty_t H^{\gamma})} \leq \|u_0\|_{(L^q_tL^r_x\cap L^\infty_t H^{\gamma})} + \||u|^p\|_{(L^{\tilde q^\prime}L^{\tilde r^{\prime})}}$$

\end{proposition}
%$$
%\|e^{i\Delta t}�\phi \|_{L^q_t L^r(\R^d)}�\lesssim \|\phi\|_{L^2(\R^d)}
%$$
%for any  pair $(q,r) $ satisfying   $\frac{2}{q} + \frac{d}{r} = \frac{d}{2}$ and $2 < q \leq \infty $. Such ordered pair is called an admissible pair . \\
For $ a\ge 0 $, $\alpha \geq 0$ and  $s\ge 0 $ we denote by $S_{a,\alpha,s}$ the linear semi-group associated with \eqref{NLSasalpha}, i.e.
$ S_{a,\alpha,s}(t)=e^{i(-\Delta)^{\alpha} t -a(-\Delta)^s t}$. It is worth noticing that  $S_{a,\alpha,s}$ is irreversible.\\
We will see in the following  proposition that that the linear semi-group $ S_{a,\alpha,s}$  enjoys the same Strichartz estimates as $ e^{i(-\Delta)^{\alpha} t}$. 
\begin{proposition}\label{LpLp}
Let $d \geq 2$,$u_0 \in H^{\gamma}(\mathbb{R}^{d})$, $\gamma \in \mathbb{R}, s\geq 0$ and $\frac{d}{2d-1} < \alpha <1$. Then  for every admissible pair $(q,r)$ satisfie the following condition:
$$
\frac{2}{q} + \frac{d}{r}=\frac{d}{2} - \gamma,~~ \frac{2}{\tilde{q}} + \frac{d}{\tilde {r}}=\frac{d}{2} + \gamma
$$ 
it holds:\\
$$\|u\|_{(L^q_t L^r_x \cap L^\infty_t H^{\gamma})} \leq \|u_0\|_{(L^q_tL^r_x\cap L^\infty_t H^{\gamma})} + \||u|^p\|_{(L^{\tilde q^\prime}L^{\tilde r^{\prime})}}$$
  
%\|S_{a,\alpha,s}(\cdot) u_0 \|_{L^q_{t>0} L^r(\R^d)}�\lesssim \|u_0\|%_{L^2(\R^d)} \, ,

%$$\|\int_0^t S_{a,\alpha,s}(t-t^\prime)f(t^{\prime})dt^{\prime}\|_{L^q_{t>0} L^r(\R^d)}\leq C\|f\|_{L^{{\tilde q}^{\prime}}_{t>0} L^{{\tilde r}^{\prime}}(\R^d)}$$,
\end{proposition}
\begin{proof} Let for any $ t\ge 0 $,  $\displaystyle{H_{a,s}(t,x) =\int e^{-ix\xi} e^{-at|\xi|^{2s}} d\xi}$, it holds 
$$
S_{a,\alpha,s}(t)\varphi =H_{a,s}(t,\cdot) \ast e^{it(-\Delta)^{\alpha}}\varphi , \quad \forall t\ge 0 \, .
$$
Noticing that for $ s>0$, $\|H_{a,s}(t,.)\|_{L^1} = \|H_{1,s}(1,.)\|_{L^1}$  and that, according to 
  Lemma 2.1 in \cite{Zhang}, $H_{1,s}(1,.) \in L^{1}(\mathbb{R}^{d})$ for $s>0$. Now let $g = |u|^p$ we have that:
\begin{align}
 \|u\|_{(L^q_tL^r_x \cap L^\infty_t H^{\gamma})}&= \|S_{a,\alpha,s}(t)(u_0) - i\int_0^t S_{a,\alpha,s}(t-t^{\prime})f(t^\prime)dt^{\prime}\|_{(L^q_tL^r_x\cap L^\infty_t H^{\gamma})}\nonumber \\ 
 &\leq  \|S_{a,\alpha,s}(t)(u_0)\|_{(L^q_t L^r_x \cap L^\infty_t H^{\gamma})} + \|\int_0^t S_{a,\alpha,s}(t-t^{\prime})f(t^\prime)\|_{(L^q_tL^r_x \cap L^\infty_tH^{\gamma})}\nonumber\\
 &=\|H_{a,s}(t,) \ast e^{it(-\Delta)^{\alpha}}(u_0)\|_{(L^q_tL^r_x \cap L^\infty_t H^{\gamma})}\nonumber\\& + \|\int_0^t H_{a,s}(t-t^{\prime})\ast e^{it(-\Delta)^{\alpha}}( f(t^{\prime}))dt^{\prime}\|_{(L^q_tL^r_x \cap L^\infty_t H^{\gamma})}\nonumber\\
 & \leq \|H_{a,s}(t,)\|_{L^1} \| e^{it(-\Delta)^{\alpha}}(u_0)\|_{(L^q_tL^r_x \cap L^\infty_t H^{\gamma})} \nonumber\\
 & + \|H_{a,s}(t-t^\prime,)\|_{L^1}\|\int_0^t H_{a,s}(t-t^{\prime}) \ast e^{it(-\Delta)^{\alpha}}( f(t^{\prime}))dt^{\prime}\|_{(L^q_tL^r_x \cap L^\infty_t H^{\gamma})}\nonumber\\
 &\leq \|u_0\|_{H^\gamma} + \| f\|_{L^{\tilde q^\prime}L^{\tilde r^{\prime}}}\nonumber
 \end{align}
 
  %\end{align*}
 
 \end{proof}
 
 With Proposition \ref{LpLp} in the hand, it is not too hard to check that  the local existence results for equation \ref{NLSasalpha} (see  Theorem 4.2 in \cite{Guo4}). More precisely, we have the following statement: 
\begin{proposition}\label{prop1}
Let $ a\ge 0$, $ \alpha \in (\frac{2d}{2d-1},1)$, $s>0$  and $u_0 \in H_{rd}^\alpha(\R^d)$ with $ d \geq 2$. There exists $ T>0 $ and a unique solution $ u\in C([0,T]; H_{rd}^\alpha)\cap L^{\frac{4\alpha}{d}+2}_T  L^{\frac{4\alpha}{d}+2} $ to \eqref{NLSasalpha} emanating from $ u_0 $. %In addition, there exists a neighborhood $ {\mathcal V}_{u_0} $
%\footnote{It is worth noticing that for $ r=0 $, the  neighborhood $ {\mathcal V}_{u_0} $ does not depend only on $\|u_0\|_{H^r} $ but on the Fourier profile of $u_0$ (see for instance \cite{Kenig})} of $ u_0 $ in $ H^r(\R^d) $ such that the associated solution map is continuous from   $ {\mathcal V}_{u_0} $ into $C([0,T]; H^r)\cap L^{\frac{4}{d}+2}(]0,T[\times\R^d) $. \\
%Finally, let $ T^* $ be the maximal time of existence of the solution $ u $ in $ H^r(\R^d) $, then 
%\begin{equation} \label{oo}
%T^*<\infty \Longrightarrow \|u\|_{L^{\frac{4}{d}+2}_{T^*}  L^{\frac{4}{d}+2} } =+\infty \; .
%\end{equation}
 \end{proposition}
 
 \section{global existence results}
 In this section, we will prove the global existence results. Let us start by the second theorem:
 \subsection{Proof of theorem \ref{theorem2}:}
 
 To prove theorem \ref{theorem2}, we will establish an a priori estimate on the Strichartz norm.

\begin{proposition}\label{LpalphaLpalpha}
Suppose that $a>0$, $\alpha >0$ and $ s>0$. Then there exists $\epsilon > 0$ such that if $ u_0$ in $H_{rd}^{\alpha}(\mathbb{R}^{d})$ and $ \|S_{a,\alpha,s}(\cdot)u_0\|_{L^{\frac{4\alpha}{d}+2}L^{\frac{4\alpha}{d}+2}} \leq \epsilon$, then the maximal time of the existence  $T^*$ of the solution emanating from $u_0$ equal to $+\infty$.
%where \\$$\theta = \frac{2\alpha(1+p)(p-1)}{2\alpha(1+p)-d(p-1)}.$$
\end{proposition}
To prove this claim, we will use the following proposition:
\begin{proposition}\label{propodeCazenave}
 There exists $\delta > 0$ with the following property. If $u_0 \in  L^2(R^ d)$ and $ T \in (0,\infty]$ are such
that $\|S_{a,\alpha,s}(.)u_0\|_{L^{p+1}([0,T],L^{p+1})} < \delta$, there exists a unique solution $u \in C([0,T],L^2(R^d))\cap L^{p+1}([0,T], L^{p+1}(R^d))$ of 
(\ref{NLSasalpha}).
In addition, $u\in L^q([0,T],L^r(R^d))$ for every admissible pair $(q,r)$, for $t \in [0,T]$.Finally, $u$
depends continuously in $C([0,T],L^2(R^d)) \cap L^{p+1}([0,T],L^{p+1}(R^d))$ on $u_0 \in L^2(R^d)$. If $u_0 \in  H^\alpha(R^d)$, then $u \in C([0,T],H^\alpha(R^d))$.
\end{proposition}
See \cite {Cazenave2} for the proof
\begin{proposition}\label{prop2}
Let $u_0 \in H^{\alpha}$ and u be the solution of \ref{NLSasalpha}. Let $T^*$ be the maximal time of the existence of u such that  $\|u\|_{L^{\frac{4\alpha}{d}+2}([0,T^*[)L^{\frac{4\alpha}{d}+2}}  < + \infty$, then $T^* = + \infty$.

\end{proposition}
\proof
 Observe that$ \|S_{a,\alpha,s}(.)u_0\|_{L^{\frac{4\alpha}{d}+2}[0,T[L^{\frac{4\alpha}{d}+2}} \longrightarrow 0$ as $T \longrightarrow  0$. Thus for
sufficiently small $T$, the hypotheses of Proposition \ref{propodeCazenave} are satisfied. Applying iteratively this proposition,
we can construct the maximal solution $u \in C([0,T^*),H^{\alpha}(\R^d)))\cap L^{\frac{4\alpha}{d}+2} ([0,T^*),L^{\frac{4\alpha}{d}+2}(\R^d))$ of (\ref{NLSasalpha}). We proceed by contradiction,
 assuming that  
$T^*< \infty$, and $\|u\|_{L^{\frac{4\alpha}{d}+2}(]0,T[,L^{\frac{4\alpha}{d}+2})} < \infty$. Let $t \in[0,T^*)$. For every $s \in [0,T^{*}-t)$ we have
$$
S(.)u(t) = u(t+.) +i\int_0^tS_{a,\alpha,s}(t-\tau)(|u|^{p-1}u)d\tau.
$$
Then by Strichartz estimate, there exists $K >0$ such that:
$$
\|S_{a,\alpha,s}(.)u(t)\|_{ L^{\frac{4\alpha}{d}+2} ([0,T^*-t),L^{\frac{4\alpha}{d}+2}(\R^d))} \leq \|u\|_{L^{\frac{4\alpha}{d}+2}(]t,T^{*}[,L^{\frac{4\alpha}{d}+2})} + K(\|u\|_{L^{\frac{4\alpha}{d}+2}(]t,T^*[,L^{\frac{4\alpha}{d}+2})})^{\frac{4}{d}+1}
$$
Therefore, for $t$ fixed close enough to $T^*$, it follows that
$$
\|S_{a,s,\alpha}(.)u(t)\|_{ L^{\sigma} ([0,T^*-t),L^{\sigma}(\R^d))} \leq \delta.
$$
Applying Proposition \ref{propodeCazenave}, we find that $u$ can be extended after $T^*$, which contradicts the maximality.\\

Now let us return to the \textbf{proof of proposition \ref{LpalphaLpalpha}}:\\

 Let $q = \frac{4\alpha(1+p)}{d(p-1)} = \frac{4\alpha}{d} + 2$ and $q^\prime = \frac{4\alpha + d}{4\alpha+2d}$, then $q^\prime$ verifies: $\frac{1}{q^\prime}=\frac{1}{q} + \frac{p-1}{\theta}$, by Holder inequalities and the Strichartz estimate we obtain:\\
 $$
 \|u\|_{L^q_TL^{p+1}} \lesssim \|u_0\|_{L^2} + \|u|u|^{p-1}\|_{L^{q^\prime}L^{\frac{p+1}{p}}} \lesssim \|u_0\|_{L^2} + \|u\|^{p}_{L^\theta L^{p+1}}\|u\|_{L^qL^{p+1}}
 $$
% Let $\tilde \theta$ such that $\frac{1}{\theta} + $ and $\gamma = \alpha$
Remark that :
$$ \frac{2\alpha}{\theta} = d(\frac{1}{2}-\frac{1}{r}),~~ \frac{2\alpha}{\theta} = d(\frac{1}{2} - \frac{1}{r})$$
Now $\frac{2}{\theta} + \frac{2d-1}{r} < d - \frac{1}{2}$ iff $ \alpha > \frac{d}{2d-1}$\\
If we take $\gamma =0$ in the Strichartz estimate this gives with Holder inequality:
\begin{align}
\|u\|_{L^\theta L^r} &\lesssim \|S_{a,\alpha,s}()u_0\|_{L^\theta L^r} + \||u|^p\|_{L^{\tilde{\prime \theta}}L^{\tilde{\prime r}}}\nonumber\\
& \lesssim \epsilon + \|u\|^{p}_{L^\theta L^r} \nonumber\\
&\lesssim \epsilon \nonumber
\end{align}
Noticing that the fractional Leibniz rule (see \cite{Kenig}) and by Holder inequality, leads to 
\begin{align}
 \||u|^{p-1}u\|_{L^{q^{\prime}}H^{\alpha,r^\prime}} &\lesssim  \|(-\Delta)^{\alpha}u\|_{L^qL^r}
\|u\|^{p-1}_{L^\theta L^r}+ \|u\|^{p-1}_{L^\theta L^r}\|u\|_{L^qL^r} \nonumber\\
& \lesssim  \|u\|^{p-1}_{L^\theta L^r} \|u\|_{L^q H^{\alpha,r}}\nonumber\\ 
& \lesssim \epsilon^{p-1}\|u\|_{L^q H^{\alpha,r}}\nonumber
\end{align}
This implies, with Strichartz estimates in the hand, that:\\
$$\|u\|_{L^q H^{\alpha,r}} \lesssim \|u_0\|_{H^\alpha} + \epsilon^{p-1}\|u\|_{L^q H^{\alpha,r}}$$ \\ this gives that : \\
$$\|u\|_{L^q H^{\alpha,r}} \lesssim \frac{\|u_0\|_{H^\alpha}}{1-\epsilon^{p-1}}$$
Now, with Strichartz estimate:\\

\begin{align}
\|u\|_{L\infty H^{\alpha}} &\lesssim \|u_0\|_{H^{\alpha}} + \|u|u|^{p-1}\|_{L^{q^{\prime}}H^{\alpha,r^\prime}}\nonumber\\
&\lesssim \|u_0\|_{H^{\alpha}}  + \epsilon^{p-1}\frac{\|u_0\|_{H^\alpha}}{1-\epsilon^{p-1}}\nonumber\\
&\lesssim \frac{1 + \epsilon^{p-1}}{1-\epsilon^{p-1}}\|u_0\|_{H^\alpha}\nonumber
\end{align}
Then for $\epsilon$ small we obtain\\
$$ \|u\|_{L_T^\infty H^{\alpha}} < \infty$$
this gives that $T = \infty$.\\

Now we are ready to \textbf{prove  Theorem \ref{theorem2}}:\\

Let $ u\in C([0,T]; H_{rd}^{\alpha}(\R^d) $ be the solution emanating from some initial datum $ u_0\in H_{rd}^{\alpha}(\R^d)  $. We have the following a priori estimates:
\begin{lemma}\label{lemma2}
Let $u \in C([0,T];H_{rd}^\alpha(\R^d))$ be the  solution of (\ref{NLSasalpha}) emanating from $ u_0\in H_{rd}^{\alpha}(\R^d) $. Then
\begin{equation}
\label{ff}
\|u\|_{L^{\infty}_T L^{2}} \leq \|u_0\|_{L^2}~~ and ~~\|(-\Delta)^{\frac{s}{2}} u\|_{{L^2_T}{L^2}} \leq \frac{1}{\sqrt{2a}} \|u_0\|_{L^2} .
\end{equation}
\end{lemma}
\begin{proof}Assume first that $ u_0 \in H^{\infty} (\R^d) $.  
Then (\ref{masse}) ensures that  the mass is decreasing  as soon as $u$ is not the null solution and  (\ref{masse}) leads to 
$$
\displaystyle{\int_0^T\|(-\Delta)^{\frac{s}{2}}u(t)\|^2_{L^2}\, dt  = -\frac{1}{2a}(\|u(T)\|_{L^2}^2 - \|u_0\|_{L^2}^2)}
\le   \frac{1}{2a }\|u_0\|_{L^2}^2.$$
This proves \eqref{ff}for smooth solutions. The result for $ u_0 \in H^{\alpha}(\R^d) $ follows by approximating $ u_0 $ in $ H^{\alpha} $ by a smooth sequence $ (u_{0,n})\subset H^\infty(\R^d) $.
\end{proof}
From the first estimate in \eqref{ff} we can obtain  that :\\
$$ \|u\|_{L^2_T L^{2}}\leq T^{\frac{1}{2}}\|u\|_{L^{\infty} L^2} \leq T^{1/2} \|u_0\|_{L^2}$$\\ and thus by interpolation:
\begin{equation}\label{est}
\|(-\Delta)^{\frac{\alpha}{2}} u\|_{L^2_T L^2}\lesssim \|(-\Delta)^{\frac{s}{2}}u\|^{\frac{\alpha}{s}}_{L^2_T L^2}\|u\|^{1-\frac{\alpha}{s}}_{L_T^{2}L^{2}} \lesssim \frac{1}{a^{\frac{\alpha}{s}}} T^{\frac{1}{2}(1-\frac{\alpha}{s})}
\end{equation}
Interpolating now between  (\ref{est}) and the first estimate of \eqref{ff} we get 
$$\|u\|_{L^{\theta}_T H^{\frac{2\alpha}{\theta}}} \lesssim \frac{1}{a^{\frac{2\alpha}{\theta s}}} T^{\frac{1}{2}(1-\frac{\alpha}{s})}~ \text{where} ~\theta = \frac{4\alpha}{d}+2$$
and the embedding  $H^{{\frac{2}{\theta}}}(\R^d) \hookrightarrow L^{ \frac{4\alpha}{d}+2}(\R^d)$ ensures that 
$$ \|u\|_{L^{\theta}_TL^{ \frac{4\alpha}{d}+2}} \lesssim \frac{1}{a^{\frac{2\alpha}{\theta s}}}T^{\frac{1}{2}(1-\frac{\alpha}{s})} \, .$$
Denoting by $ T^* $ the  maximal time of existence of $ u $ and letting $ T $ tends to $ T^* $, this contradicts proposition \eqref{prop2} whenever $ T^* $ is finite.
 This proves that the solutions are global in $ H^{\alpha}(\R^d) $. 
 
 %Moreover, for $ s=1 $ we have $  \|u\|_{L^{\frac{4}{d}+2}_TL^{\frac{4}{d}+2}}\lesssim 1 $ for any $ T>0 $ which ensures that 
 %$$
 %\|u\|_{L^{\frac{4}{d}+2}(\R_+^*\times\R^d) }\lesssim 1 \; .
 %$$
 \begin{remark}\label{normecrtique}
 Note that for $s=\alpha$, we have that $\|u\|_{L^{\frac{4\alpha}{d}+2}_TL^{\frac{4\alpha}{d}+2}}\lesssim 1 $ for any $ T>0 $ which show  that 
 $$
 \|u\|_{L^{\frac{4\alpha}{d}+2}(\R_+^*\times\R^d) } < +\infty \;
 $$
 In plus,
 $$\|u\|_{L^{\frac{4\alpha}{d}+2}(\R_+^*\times\R^d) } \longrightarrow 0, ~ \text{when}~ a \longrightarrow +\infty$$
 \end{remark}
\subsection{Proof Theorem \ref{theorem1}}\label{sect}
Now we will prove the global existence for small data, to do this we will use the following fractional Galgliardo-Niremberg inequalities (see \cite{Molinet}):
\begin{lemma}\label{Gagliardo}
 Let $q$, $r$ be  any real numbers satisfying $1\leq q$, $ r \leq \infty$, and $s$, $s_1$ be two reals numbers. If $u$ is any functions in
$C^{\infty}_{0}(\R^d)$, then
$$
\|D^s u\|_{L^p} \leq C \|D^{s_1} u\|^a_{L^r} \|u\|_{L^q}^{1-a}
$$
where 
$$
\frac{1}{p} = \frac{s}{d} + a(\frac{1}{r}-\frac{s_1}{d}) + (1-a) \frac{1}{q},
$$
for all $a$ in the interval
$$
\frac{s}{s_1} \leq a \leq 1,
$$
where $C$ is a constant depending only on $ d$, $s$, $s_1$,$q$,$r$ and $a$.
\end{lemma}
Now we have the following one:
%Note that  the global existence for any $ u_0\in L^2(\R^d) $  with $ \|u_0\|_{L^2} $ small enough can be proven, as for the critical NLS equation,  directly by a fixed point argument thanks to Lemma \ref{LpLp}.  This ensures the global existence in $ H^r(\R^d) $, $r\ge 0 $, under the same smallness condition 
% on $ \|u_0\|_{L^2} $. We will not invoke this fact here  and we will directly prove 
%Assertion 1 of Theorem \ref{theorem4} 	by combining \eqref{ZZ} and a  monotony result on  $t\mapsto  E(u(t)) $. 
%To do this, we will  work with smooth solutions and then get the result for $ H^1$-solutions by continuity with respect to initial data. 
\begin{proposition}\label{propenergie}
Let  $ \alpha > 0$ and $u \in H^{\alpha}(\mathbb{R}^{d})$. Then there exists $C = C(d,\alpha)$ such that:
$$
\|(-\Delta)^{\frac{\alpha}{2}}u\|^2\big(\frac{1}{2}-C\|u\|_{L^2}^{\frac{4\alpha}{d}}\big) \leq E(u(t)) 
$$
\end{proposition}
\textbf{Proof of proposition \ref{propenergie}:}\\
We have that :
$$
E(u(t)) = \frac{1}{2}\|(-\Delta)^{\frac{\alpha}{2}}u\|^2_{L^2} - \frac{d}{4\alpha + 2d} \|u\|_{L^{\frac{4\alpha}{d}+2}}^{\frac{4\alpha}{d}+2}
$$
But by the fractional Gagliardo-Niremberq inequality there exists $A=A(\alpha,d)$ such that:\\

$$
\displaystyle{\int |u|^{\frac{4\alpha}{d}+2}dx \leq A (\int|(-\Delta)^{\frac{\alpha}{2}}u|^2dx)(\int|u|^2dx)^{\frac{2\alpha}{d}}}
$$
then 
\begin{align}
E(u(t)) &\geq \frac{1}{2}\|(-\Delta)^{\frac{\alpha}{2}}u\|^2_{L^2} - A\frac{d}{4\alpha + 2d}(\int|(-\Delta)^{\frac{\alpha}{2}}u|^2dx)(\int|u|^2dx)^{\frac{2\alpha}{d}}\nonumber\\
& = \|(-\Delta)^{\frac{\alpha}{2}}u\|^2_{L^2}\bigg(\frac{1}{2} - A\frac{d}{4\alpha + 2d}(\int|u|^2dx)^{\frac{2\alpha}{d}}\bigg).\nonumber
\end{align}
Now let us return to the proof of theorem \ref{theorem1}:\\
Let  $ u\in C([0,T]; H^\infty(\R^d)) $ be  a solution to \eqref{NLSasalpha} emanating from $u_0\in H^\infty(\R^d) $. 
Then it holds 
$$\frac{d}{dt}E(u(t)) =  \displaystyle{-a\int |(-\Delta)^{\frac{s+\alpha}{2}}u(t)|^2 +a Im \int \big((-\Delta)^s u(t)\big)\,  |u(t)|^{p-1}\overline{u(t)}}$$
and H\"older inequalities in physical space and in Fourier space lead to
$$
\Bigl| \int ((-\Delta)^s u) \, |u|^{p+1} \Bigr| \leq \|(-\Delta)^s u \|_{L^2}\|u\|_{L^{2p+2}}^{p+1}
$$
with
$$
\|(-\Delta)^s u \|_{L^2} \leq \|(-\Delta)^{\frac{s+\alpha}{2}}u\|_{L^2}^{\frac{2s}{s+\alpha}}\|u\|_{L^2}^{\frac{\alpha-s}{\alpha+s}}\, .
$$
Using Gagliardo-Nieremberg inequality, we obtain $$
\|u\|_{L^{\frac{8\alpha}{d}}+2}^{\frac{8\alpha}{d} + 2} \leq C_d^{\frac{8 \alpha}{d}+2} \|\nabla u\|_{L^2}^{4\alpha}\|u\|^{{\frac{8\alpha}{d} + 2}-4\alpha}_{L^2}\, .
$$
 This estimate together with Cauchy-Schwarz inequality (in Fourier space)
 $$
\|\nabla u\|_{L^2} \leq \|(-\Delta)^{\frac{s+\alpha}{2}} u \|_{L^2}^{\frac{1}{s+\alpha}}\|u\|_{L^2}^{\frac{(s+\alpha-1)}{s+\alpha}}
$$

lead to 
$$
\|u\|_{L^{\frac{8\alpha}{d}}+2}^{\frac{4\alpha}{d} + 1} \leq C_d^{\frac{4\alpha}{d} + 1}\|(-\Delta)^{\frac{s+\alpha}{2}} u \|_{L^2}^{\frac{2\alpha}{s+1}}\|u\|_{L^2}^{\frac{2s}{s+1}}\| u\|_{L^2}^{\frac{4\alpha}{d}+1 -\frac{2\alpha}{s+\alpha}}\, .
$$
Combining the above estimates we eventually obtain
$$
\frac{d}{dt}E(u(t)) \leq a \|(-\Delta)^{\frac{s+\alpha}{2}}u\|_{L^2}^2(C_d^{\frac{4\alpha}{d} + 1}\|u\|_{L^2}^{4\alpha/d} - 1)
$$
which together with \eqref{ff} implies that $ E(u(t)) $ is not increasing for $ \|u_0\|_{L^2} \le\frac{1}{C_d^{1+\frac{d}{4}}} $ implies $E(u(t)) \leq E(u_0)$ for all $ t \geq 0$. Now with proposition \ref{propenergie} in the hand we obtain that $\|(-\Delta u)^{\frac{\alpha}{2}}\|_{L^2} \lesssim E(u_0)$, for $\|u_0\|_{L^2}$ small enough.This finishes the proof. \\

%Now we will prove the global existence for  small data, for this we will use the following generalized Gagliardo-Niremberg inequalities (see for instance \cite{Friedman}):
%\begin{lemma}\label{Friedman}
% Let $q$, $r$ be  any real numbers satisying $1\leq q$, $ r \leq \infty$, and let $j$, $m$ be any integers satisfying $0\leq j <m$. If $u$ is any functions in 
%$C^{m}_{0}(\R^d)$, then 
%$$
%\|D^j u\|_{L^s} \leq C \|D^m u\|^a_{r} \|u\|_q^{1-a}
%$$
%where\\ 
%$$
%\frac{1}{s} = \frac{j}{d} + \mu(\frac{1}{r}-\frac{m}{d}) + (1-\mu) \frac{1}{q},
%$$
%for all $\mu$ in the interval
%$$
%\frac{j}{m} \leq \mu \leq 1,
%$$
%where $C$ is a constant depending only on $ d$, $m$, $j$,$q$,$r$ and $\mu$.
%\end{lemma}
\textbf{Proof of theorem \ref{third theorem}:} The second part of this theorem was proved previously (see remark \ref{normecrtique}). Let us prove the scattering:\\
% note that we will treat only the
%positive time  $t \longrightarrow +\infty$ , since the negative time can be treated in the same way.\\ 
Let $v(t):=S_{-a,-\alpha,-s}(t)u(t)$ then
\begin{equation}
 v(t) = u_{0} + i\int_{0}^{t}S_{a,\alpha,s}(s)(|u|^{\frac{4\alpha}{d}}u(s))ds. \nonumber
\end{equation} 
Therefore for $0 < t <\tau$,
$$ 
v(t)-v(\tau) = i\int_{\tau}^{t}S_{a,\alpha,s}(-t^{\prime})\big(|u|^{\frac{4\alpha}{d}}u\big)dt^{\prime}.
$$
It follows from Strichartz's estimates (as previously) that:
$$
\|v(t)-v(\tau)\|_{L^2} = \|i\int_{\tau}^{t}S_{a,\alpha,s}(-t^{\prime})\big(|u|^{\frac{4\alpha}{d}}u\big)\|_ {L^2}\leq \|u\|^{\frac{4\alpha}{d}+1}_{L^{\frac{4\alpha}{d} +2}(t,\tau) L^{\frac{4\alpha}{d} +2}}$$
But  by remark (\ref{normecrtique}), for $s=\alpha$ we have that $u\in L^{\frac{4\alpha}{d}+2} ((0,\infty),L^{\frac{4\alpha}{d}+2})$, then
%\leq C \parallel u\parallel_{L^{\frac{4\alpha}{d}+2}([t,\tau]\times\mathbb{R}^{d})}^{\frac{4\alpha}{d}+1}.
%$$
the right hand side goes to zero when $t, \tau \longrightarrow + \infty$. The scattering follows from the Cauchy criterion.\\

\end{document}